\documentclass{amsart}
\usepackage{fancyhdr}
\usepackage{latexsym}
\usepackage{amssymb}
\newtheorem{lemma}{Lemma}[section]

\newtheorem{theorem}[lemma]{Theorem}
\theoremstyle{definition}
\newtheorem{definition}[lemma]{Definition}
\theoremstyle{definition}

\newtheorem{remark}[lemma]{Remark}
\theoremstyle{definition}

\title {Lattice pseudo-effect algebras as  double residuated structures}
\author{David J. Foulis, Sylvia Pulmannov\'a and Elena Vincekov\'a}\thanks{The second and third author were supported
Center of Excellence SAS -~Quantum Technologies;
ERDF OP R\&D Projects CE QUTE ITMS 26240120009,  and meta-QUTE ITMS 26240120022;
the grant VEGA No. 2/0032/09 SAV;
the Slovak Research and Development Agency under
the contract LPP-0199-07}
\begin{document}
\address{Department of Mathematics and Statistics, University of Massachusetts, Amherst, MA, USA; Mathematical Institute, Slovak Academy of Sciences, \v
Stef\'anikova 49, 814 73 Bratislava, Slovakia }
\email{foulis@math.umass.edu, pulmann@mat.savba.sk, vincek@mat.savba.sk}
\maketitle
\markboth{Foulis, D., Pulmannov\'a, S., Vincekov\'a, E.}{LPEAs as double residuated structures}
\date{}



\begin{abstract} Pseudo-effect algebras are partial algebraic structures, that were introduced as a non-commutative generalization of effect algebras.  In the present paper, lattice ordered pseudo-effect algebras are considered as possible algebraic non-commutative analogs of non-commutative non-standard reasoning. To this aim, the interplay among conjunction, implication and negation connectives is studied. It turns out that in the non-commutative reasoning, all these connectives are doubled. In particular, there are two negations and two pairs consisting of conjunction and implication, related by residuation laws. The main result of the paper is a characterization of lattice pseudo-effect algebras in terms of so-called pseudo Sasaki algebras.  We also show that all pseudo-effect algebras can be characterized in terms of certain partially defined double residuated structures.

\end{abstract}
\maketitle

\section{Introduction}

An effect algebra is a partial algebraic structure, originally introduced as an algebraic base for unsharp quantum measurements. Recently, in \cite{FoPu},  lattice effect algebras (LEAs) have been studied as possible algebraic models for the semantics of non-standard symbolic logic, just as MV-algebras (special kind of LEAs) are algebraic models for Lukasiewicz many-valued logics, and orthomodular lattices (also a special kind of LEAs) are algebraic models for sharp quantum logical calculi. In particular, the interplay among conjunction, implication, and negation connectives on LEAs  has been studied, where the conjunction and implication connectives are related by a residuation law. As a main result, a characterization of LEAs has been obtained in terms of so-called Sasaki algebras. A Sasaki algebra is a structure $(P;\leq;0,1, ^{\perp}, .)$ consisting of a bounded poset equipped with a unary operation, the involution $a\mapsto a^\perp$ and a binary operation $(p,q)\mapsto p.q$ called the Sasaki product, which plays the role of a conjunction connective.

In \cite{DvVe1, DvVe2} pseudo-effect algebras were introduced as a non-commutative generalization of effect algebras. In the present paper, we extend the study of logical aspects to lattice pseudo-effect algebras (LPEAs). In \cite{PuSas}, it was shown that there are two analogues of the Sasaki product in  LPEAs: the "right" and the "left".  It turns out that all connectives are doubled: we need to consider two conjunctions, two implications and two negations, and we also have two residuation laws. As a  main result, we obtain a characterization of lattice pseudo-effect algebras in terms of so-called pseudo Sasaki algebras. We also obtain  characterizations of some special subclasses of LPEAs, among them pseudo MV algebras, in terms of additional identities for logical connectives.

In \cite{ChaHa},  effect algebras are characterized as so-called conditionally residuated structures,
in which the residuated operations are partially defined. The latter characterization is extended to a subclass of pseudo-effect algebras, so-called good pseudo-effect algebras. In the end of the present paper, we characterize all pseudo-effect algebras in terms of some double residuated structures with partially defined operations.

\section{Conjunction, implication, and negation connectives}

In \cite{FoPu}, the notion of a conjunction/implication poset (CI-poset) is introduced as a system $(P;\leq; 0,1,.,\rightarrow)$ consisting of a bounded poset $(P;\leq; 0,1)$ equipped with two binary compositions $.$ and $\rightarrow$ called the conjunction connective and the implication connective, respectively, satisfying the \emph{unity law}: $1.p=p.1=p$, and the \emph{residuation law}: $p.q\leq r \, \Leftrightarrow \, q\leq (p\rightarrow    r)$.
Owing to residuation law, the binary mappings $\to$ and $.$ determine each other uniquely.
Several "logical laws" are then considered, some of them are satisfied by every CI-poset, the others might or might not be satisfied.

Now we will consider a structure $(P;\leq,0,1,\circ,\rightarrow,\ast,\leadsto)$ satisfying the following two axioms:
\begin{itemize}
\item[(1)] $1\circ a=a\circ 1=a=1\ast a=a\ast 1$ (unity)
\item[(2)] $a\circ b\leq c \,\Leftrightarrow\, b\leq a\leadsto c$ and $a\ast b\leq c \,\Leftrightarrow\, b\leq a\rightarrow c$ (residuation)
\end{itemize}

and with complements defined by
$a^-:=a\rightarrow 0\;\;\;\;\;\;\;\; a^{\sim}:=a\leadsto 0$

\begin{definition}
The structure $(P,;\leq,\circ,\ast,\rightarrow,\leadsto,0,1)$ satisfying axioms (1) and (2) will be called a \emph{double CI-poset}. If a double CI-poset is a lattice, we call it a \emph{double CI-lattice}.
\end{definition}

In what follows, $(P;\leq,\circ,\ast,\rightarrow,\leadsto,0,1)$ is a double CI-poset.

\begin{theorem}
$a,b\in P$, $a\leq b \,\Leftrightarrow\, a\rightarrow b=1$ ($\,\Leftrightarrow\, a\leadsto b=1$) (deduction law).
\end{theorem}

\begin{proof}
$a\leq b \,\Leftrightarrow\, a\circ 1\leq b \;\,(a\ast 1\leq b)$ which is by residuation law equivalent with $1\leq a\leadsto b \;\,(1\leq a\rightarrow b)$.
\end{proof}

\begin{lemma}\label{le:2.3}
Let $a,b$ be elements of $P$. Then
\begin{itemize}
\item[(i)] $a\rightarrow 1=0\rightarrow a=a\rightarrow a=1$; $a\leadsto 1=0\leadsto a=a\leadsto a=1$
\item[(ii)] $a^{\sim}=1 \,\Leftrightarrow\, a=0\,\Leftrightarrow\, a^-=1$ and $0^{{\sim}-}=0^{-{\sim}}=0$
\item[(iii)] $a\circ b=0 \,\Leftrightarrow\, b\leq a^{\sim}$ and $a\ast b=0 \,\Leftrightarrow\, b\leq a^-$
\item[(iv)] $a^-\circ a=0 \,\Leftrightarrow\, a^{-{\sim}}\geq a$ and $a^{\sim}\ast a=0 \,\Leftrightarrow\, a^{{\sim}-}\geq a$
\item[(v)] $b\leq a\leadsto (a\circ b)$ and $b\leq a\rightarrow (a\ast b)$
\item[(vi)] $b=1\rightarrow b=1\leadsto b$
\item[(vii)] $a\circ b\leq a$ and $a\ast b\leq a$
\item[(viii)] $a\circ 0=0\circ a=0=a\ast 0=0\ast a$
\item[(ix)] $a\circ b=1 \,\Leftrightarrow\, a=b=1 \,\Leftrightarrow\, a\ast b=1$
\end{itemize}
\end{lemma}

\begin{proof}
(i): As $a\leq 1; 0\leq a$ and $a\leq a$, we  use just deduction law to get the result.\\
(ii): We use definition and deduction law again to get $0^-=0\rightarrow 0=1=0\leadsto 0=0^{\sim}$. On the other hand, if $a^-=1$, then $1\leq a\rightarrow 0$, whence $a=a\ast 1\leq 0$ by residuation, so $a=0$. Also if $a^{\sim}=1$, then $1\leq a\leadsto 0$, whence $a=a\circ 1\leq 0$ and so $a=0$.\\
(iii), (iv) and (v) directly follow from residuation.\\
(vi): When we put $a:=1$ in (v), we obtain $b\leq 1\rightarrow b$ and $b\leq 1\leadsto b$. As $1\rightarrow b\leq 1\rightarrow b$ and $1\leadsto b\leq 1\leadsto b$, from residuation we get $1\rightarrow b=1\ast (1\rightarrow b)\leq b$ and $1\leadsto b=1\circ (1\leadsto b)\leq b$.\\
(vii): We use residuation law and (i) to get $a\circ b\leq a$ iff $b\leq a\leadsto a=1$ and $a\ast b\leq a$ iff $b\leq a\rightarrow a=1$.\\
(viii): $a\circ 0=a\ast 0=0$ is a consequence of residuation law and the remaining is implied by (vii) when $a=0$.\\
(ix): If $a\circ b=1$ or $a\ast b=1$, then by (vii) $a=1$. Now $b=1$ follows from unity law. The converse is obvious.
\end{proof}

\begin{theorem}\label{th:prop}
If $a,b,c\in P$ and $(b_i)_{i\in I}$ is a family of elements of $P$, then
\begin{itemize}
\item[(i)] $a\circ (a\leadsto b)\leq b$ and $a\ast (a\rightarrow b)\leq b$ (modus ponens)
\item[(ii)] $b\leq c\,\Rightarrow\, (a\leadsto b)\leq (a\leadsto c),\; (a\rightarrow b)\leq (a\rightarrow c)$ (right monotone law for the consequent)
\item[(iii)] $a^-\leq a\rightarrow b$ and $a^{\sim}\leq a\leadsto b$ (Duns Scotus' law)
\item[(iv)] if $\bigvee_{i\in I} b_i$ exists in $P$, then $\bigvee_{i\in I}a\circ b_i$ and $\bigvee_{i\in I}a\ast b_i$ exist in $P$ and $a\circ\bigvee_{i\in I} b_i=\bigvee_{i\in I} a\circ b_i$, $a\ast\bigvee_{i\in I} b_i=\bigvee_{i\in I} a\ast b_i$ (preservation of supremum by conjunction)
\item[(v)] if $\bigwedge_{i\in I} b_i$ exists in $P$, then $\bigwedge_{i\in I} (a\rightarrow b_i)$ and $\bigwedge_{i\in I} (a\leadsto b_i)$ exist in $P$ and $a\rightarrow \bigwedge_{i\in I} b_i=\bigwedge_{i\in I} (a\rightarrow b_i)$, $a\leadsto \bigwedge_{i\in I} b_i=\bigwedge_{i\in I} (a\leadsto b_i)$ (preservation of infimum in the consequent)
\item[(vi)] $b\leq c$ implies $a\circ b\leq a\circ c$ and $a\ast b\leq a\ast c$ (right monotone law for conjunction)
\end{itemize}
\end{theorem}

\begin{proof}
(i): By residuation $a\circ (a\leadsto b)\leq b$ iff $a\leadsto b\leq a\leadsto b$ and $a\ast (a\rightarrow b)\leq b$ iff $a\rightarrow b\leq a\rightarrow b$.\\
(ii): If $b\leq c$, then by modus ponens $a\circ (a\leadsto b)\leq b\leq c$ and by residuation law $a\leadsto b\leq a\leadsto c$. Similarly $a\ast (a\rightarrow b)\leq b\leq c$ and so $a\rightarrow b\leq a\rightarrow c$.\\
(iii): By (ii) $a^-=a\rightarrow 0\leq a\rightarrow b$ and $a^{\sim}=a\leadsto 0\leq a\leadsto b$.\\
Parts (iv) and (v) follow from the fact that $(\circ, \leadsto)$ and $(\ast, \rightarrow)$ are residuated mappings \cite{BJ}.\\
(vi): Suppose $b\leq c$. Then $a\circ c=a\circ (b\vee c)=(a\circ b)\vee (a\circ c)\geq a\circ b$ by (iv). Proof for the second operation goes the same way.
\end{proof}

\begin{definition}\label{de:pisad} A double CI-poset $P$ satisfies
\begin{enumerate}
\item[{\rm(i)}] \emph{pseudo-involution law} iff $a^{- \sim}=a={a^{\sim}}^-$ and $a\leq b\ \Rightarrow \ b^-\leq a^-, b^{\sim}\leq a^{\sim}$
\item[{\rm(ii)}] \emph{divisibility law} iff $c\leq a, c\leq b$ $\Leftrightarrow$ $c\leq a\circ(a\leadsto b)= a\ast (a\rightarrow b)$
\item[{\rm(iii)}] \emph{ortho-exchange law} iff\\
$a^-\circ b^-=0$ and $c^{\sim}\leq a\circ b$ implies $b^-\leq a\ast c$\\
and\\
$a^{\sim}\ast b^{\sim}=0$ and $c^-\leq a\ast b$ implies $b^{\sim}\leq a\circ c$
\item[{\rm(iv)}] \emph{self-adjointess law} iff $a\circ b\leq c \, \Leftrightarrow\, a\circ c^-\leq b^-$ and $a\ast b\leq c\, \Leftrightarrow a\circ c^{\sim}\leq b^{\sim}$
\end{enumerate}
\end{definition}

\begin{lemma}\label{le:divis} A CI-poset $P$  satisfies the divisibility law iff $a\wedge b=a\ast(a\rightarrow b)=a\circ(a\leadsto b)$
\end{lemma}
\begin{proof} By (vii) from Lemma \ref{le:2.3}, $a\circ(a\leadsto b)\leq a$, $a\ast(a\rightarrow b)\leq a$. By modus ponens (Theorem \ref{th:prop} (i)), $a\circ (a\leadsto b)\leq b$, $a\ast(a\rightarrow b)\leq b$.
Then divisibility implies  the desired statement.
\end{proof}

\begin{lemma}\label{le:prop}
Suppose that $^{\sim}$ and $^-$ form a pseudo-involution on $P$ and the self-adjointness law is satisfied. Then
\begin{itemize}
\item[(i)] $a\rightarrow b=(a\circ b^{\sim})^-$ and $a\leadsto b=(a\ast b^-)^{\sim}$
\item[(ii)] $P$ satisfies the divisibility law iff it satisfies ortho-exchange law, that is:
\begin{itemize}
\item[(1)] $c\leq a,b\;\Leftrightarrow\; c\leq a\ast (a\rightarrow b)=a\ast (a\circ b^{\sim})^-$\\ iff $(a^-\circ b^-=0 \mbox{ and } c^{\sim}\leq a\circ b\;\Rightarrow\; b^-\leq a\ast c)$
\item[(2)] $c\leq a,b \;\Leftrightarrow\; c\leq a\circ (a\leadsto b)=a\circ (a\ast b^-)^{\sim}$\\ iff $(a^{\sim}\ast b^{\sim}=0 \mbox{ and } c^-\leq a\ast b\;\Rightarrow\; b^{\sim}\leq a\circ c)$
\end{itemize}
\end{itemize}
\end{lemma}

\begin{proof}
(i): Using residuation and self-adjointness we stepwise get $c\leq (a\leadsto b)\,\Leftrightarrow\, a\circ c\leq b\,\Leftrightarrow\, a\ast b^-\leq c^-\,\Leftrightarrow\, c\leq (a\ast b^-)^{\sim}$ and $c\leq (a\rightarrow b)\,\Leftrightarrow\, a\ast c\leq b\,\Leftrightarrow\, a\circ b^{\sim}\leq c^{\sim}\,\Leftrightarrow\, c\leq (a\circ b^{\sim})^-$, therefore $a\leadsto b=(a\ast b^-)^{\sim}$ and $a\rightarrow b=(a\circ b^{\sim})^-$.\\
(ii): We make a proof of the first statement, the other one is then straightforward variation. Suppose that $P$ satisfies the divisibility law and let $a^-\circ b^-=0$ and $c^{\sim}\leq a\circ b$. Then $b^-\leq a^{-{\sim}}=a$  by Lemma \ref{le:2.3} (iii), and $(a\circ b)^-\leq c^{{\sim}-}=c$. Then $b^-=a\wedge b^-=a\ast (a\rightarrow b^-)=a\ast (a\circ b^{-{\sim}})^-=a\ast (a\circ b)^-\leq a\ast c$, where we use part (i) and Theorem \ref{th:prop} (vi).

 Conversely, let the ortho-exchange law be satisfied in $P$ and let $c\leq a,b$. We set $d:=a\rightarrow b=(a\circ b^{\sim})^-$. Then $d^{\sim}=a\circ b^{\sim}$ and as $c\leq b$, we get $b^{\sim}\leq c^{\sim}$, so that $a\circ b^{\sim}\leq a\circ c^{\sim}$, which is in fact $d^{\sim}\leq a\circ c^{\sim}$. It is also true that $c\leq a\,\Rightarrow\, a^-\circ c\leq a^-\circ a=a^-\circ a^{-{\sim}}=0$ by Lemma \ref{le:2.3} (iv) because
  $a\leq a$. But then $a^-\circ c=0$. Now we apply ortho-exchange law to obtain $c\leq a\ast d=a\ast (a\circ b^{\sim})^-$.
\end{proof}

\section{From pseudo Sasaki algebra to LPEA}

\begin{definition}\label{de:psSas} A structure $(P;^-,^{\sim},\circ,\ast,0,1)$
will be called a \emph{pseudo Sasaki algebra} if it satisfies the following axioms:
\begin{itemize}
\item[(1)] $a^{-\sim}=a^{\sim -}=a$, $\;\; a\leq b\,\Rightarrow\,b^-\leq a^-, b^{\sim}\leq a^{\sim}$ (pseudo-involution)
\item[(2)] $a\circ 1=1\circ a=a=1\ast a=a\ast 1$ (unity)
\item[(3)] $a\circ b\leq c$ iff $a\ast c^-\leq b^-$\\
           $a\ast b\leq c$ iff $a\circ c^{\sim}\leq b^{\sim}$ (self-adjointness)
\item[(4)] $c\leq a, c\leq b\,\Rightarrow\, c\leq a\circ (a\ast b^-)^{\sim}=a\ast (a\circ b^{\sim})^-$ (divisibility)
\item[(5)] $a\leq b^-,c\leq a^{\sim}\circ b^{\sim}\,\Rightarrow\, (a^{\sim}\circ b^{\sim})\circ c^{\sim}=a^{\sim}\circ (b^{\sim}\circ c^{\sim})$; and
        $b\leq a^{\sim}, c\leq b^-\ast a^-\ \Rightarrow \ c^-\ast (b^-\ast a^-)=(c^-\ast b^-)\ast a^-$
     (partial associativity)
\item[(6)] $a\leq b^-$ $\implies$ $(a^{\sim}\circ b^{\sim})^-=(b^-\ast a^-)^{\sim}$
\end{itemize}
\end{definition}

\begin{lemma}\label{le:1} Let $(P;^-, ^{\sim},\circ,\ast,0,1)$ be a pseudo Sasaki algebra.
For every $a,b,c\in P$ the following hold:
\begin{itemize}
\item[(a)] $a^-\circ a=a\circ a^{\sim}=0=a\ast a^-=a^{\sim}\ast a$
\item[(b)] $a\circ 0=0\circ a=0=a\ast 0=0\ast a$
\item[(c)] $a\leq b^-$ iff $a\circ b=0$\\
           $a\leq b^{\sim}$ iff $a\ast b=0$
\item[(d)] $a\circ b\leq a$\\
      $a\ast b\leq a$
\end{itemize}
\end{lemma}

\begin{proof}
(a): It is clear from axiom (1), that $1^-=1^{\sim}=0$, $0^-=0^{\sim}=1$. By the third axiom $a\circ a^{\sim}\leq 0$ iff $a\ast 0^-\leq a$ which is true by (2). Also $a^-\circ a\leq 0$ iff $a^-\ast 0^-\leq a^-$, $a\ast a^-\leq 0$ iff $a\circ 0^{\sim}\leq a$ and $a^{\sim}\ast a\leq 0$ iff $a^{\sim}\circ 0^{\sim}\leq a^{\sim}$, where all of the latter conditions hold.\\
(b): By (3) $a\circ 0\leq 0$ if and only if $a\ast 0^-\leq 0^-$ and while $a\ast 0^-=a$ and $0^-=1$, we get a true statement. On the other hand $0\circ a\leq 0$ iff $0\ast 0^-\leq a^-$. But $0\ast 0^-$ is 0 by (a) and so $0\ast 0^-\leq a^-$ holds again. Proof for $a\ast 0$ and $0\ast a$ is similar.\\
(c): $a\leq b^- \Leftrightarrow a\ast 1\leq b^- \Leftrightarrow a\ast 0^-\leq b^-$ and self-adjointness now implies that $a\circ b\leq 0$ which is iff $a\circ b=0$. Similarly $a\leq b^{\sim} \Leftrightarrow a\circ 1\leq b^{\sim} \Leftrightarrow a\circ 0^{\sim}\leq b^{\sim} \Leftrightarrow a\ast b\leq 0 \Leftrightarrow a\ast b=0$.\\
(d): $a\circ b\leq a$ iff $a\ast a^-\leq b^-$, but by (a) $a\ast a^-=0$, so that the inequality is true. Similarly $a\ast b\leq a$ iff $a\circ a^{\sim}\leq b^{\sim}$ where $a\circ a^{\sim}=0$.
\end{proof}

\begin{theorem}\label{th:CIS} Every pseudo Sasaki algebra is a double CI-lattice. Conversely, a double CI-poset is a pseudo Sasaki algebra iff it has a pseudo-involution, self-adjoiness, divisibility (ortho-excahnge), and satisfies  conditions (5) and (6) from Definition \ref{de:psSas}.
\end{theorem}
\begin{proof}  Let $(P;^-,^{\sim},\circ,\ast,0,1)$ be a pseudo Sasaki algebra. For $a,b\in P$, define $a\rightarrow b:=(a\circ b^{\sim})^-$ and $a\leadsto b:=(a\ast b^-)^{\sim}$.
Now consider the structure \newline $(P;\leq,0,1, \circ, \ast, \rightarrow, \leadsto)$. The unity law $1\circ a=a\circ 1=1\ast a=a\ast 1=a$ is satisfied by (2) of Definition \ref{de:psSas}.
By (3) and (1) of Definition \ref{de:psSas}, $a\circ b\leq c$ iff $a\ast c^-\leq b^-$ iff $b\leq (a\ast c^-)^{\sim}=a\leadsto c$. Similarly, $a\ast b\leq c$ iff $a\circ c^{\sim}\leq b^{\sim}$ iff $b\leq (a\circ c^{\sim})^-=a\rightarrow c$. This proves residuation and therefore
$P$ is a double CI-poset. Condition (4) of Definition \ref{de:psSas} can be rewritten as $c\leq a, c\leq b$ $\implies $ $c\leq a\circ(a\leadsto b)=a\ast (a\rightarrow b)$. By Lemma \ref{le:divis}, $a\wedge b= a\circ(a\leadsto b)=a\ast (a\rightarrow b)$.
Condition (1) of Definition \ref{de:psSas} yields $a\vee b=(a^-\wedge b^-)^{\sim}=(a^{\sim}\wedge b^{\sim})^-$. It follows that $P$ is a CI-lattice.

Conversely, let $(P;\leq,0,1, \circ, \ast, \rightarrow, \leadsto)$ be a CI-poset that has a pseudo-involution, self-adjointness, divisibility and satisfies conditions (5) and (6) from Definition \ref{de:psSas}. By Lemma \ref{le:prop}, $a\rightarrow b=(a\circ b^{\sim})^-$, and $a\leadsto b= (a\ast b^-)^{\sim}$.  With $a^-=a\rightarrow 0$, $a^{\sim}=a\leadsto 0$, we obtain a pseudo Sasaki algebra.
\end{proof}

\begin{definition}
A {\it pseudo-effect algebra}  (PEA) is a partial algebra $(P;\oplus,0,1)$ of the type $(2,0,0)$ where the following axioms hold for any $a,b,c\in P$:
\begin{itemize}
\item[(PE1)] $a\oplus b$ and $(a\oplus b)\oplus c$ exist iff $b\oplus c$ and $a\oplus (b\oplus c)$ exist and in this case $(a\oplus b)\oplus c=a\oplus (b\oplus c)$
\item[(PE2)] there exists exactly one $d\in E$ and exactly one $e\in E$ such that $a\oplus d=e\oplus a=1$
\item[(PE3)] if $a\oplus b$ exists, there are elements $d,e\in E$ such that $a\oplus b=d\oplus a=b\oplus e$
\item[(PE4)] if $a\oplus 1$ or $1\oplus a$ exists, then $a=0$
\end{itemize}
\end{definition}

If the pseudo-effect algebra is lattice ordered, we speak about lattice pseudo-effect algebra or LPEA. In a pseudo-effect algebra, we may define a partial order in the following way:
$$a\leq b  \mbox{ iff } c\oplus a=b \mbox{ for some } c\in E.$$
Equivalently,
$$
a\leq b  \mbox{ iff } a\oplus d=b \mbox{ for some } d\in E.
$$
These equivalences determine two partial subtractions ("right" and "left") $\diagdown$ and $\diagup$ like this: $b\diagdown a$ is defined and equals $x$ iff $b=x\oplus a$, and $a\diagup b$ is defined and equals $y$ iff $b=a\oplus y$.
Thus both $b\diagdown a$ and $a\diagup b$ are defined iff $a\leq b$, and then $(b\diagdown a)\oplus a=b=a\oplus (a\diagup b)$. Moreover, for the elements $d$ and $e$ in axiom (PE2) we write
$1\diagdown a=:a^{-}$ (the "left" complement) and $a\diagup 1=a^{\sim}$ (the "right" complement). Clearly, $0^{\sim}=1=0^-$ and $1^{\sim}=0=1^-$.  Notice that $a\oplus b$ exists iff $a\leq b^{-}$, equivalently, iff $b\leq a^{\sim}$.

\begin{lemma}\label{le:minus}
If $a\leq b$, then
\begin{itemize}
\item[(a)] $a\diagup b=a^{\sim}\diagdown b^{\sim}$
\item[(b)] $b\diagdown a=b^-\diagup a^-$
\end{itemize}

If $a\leq b^-$, then
\begin{itemize}
\item[(c)] $(a\oplus b)^-=b^-\diagdown a$
\item[(d)] $(a\oplus b)^{\sim}=b\diagup a^{\sim}$
\end{itemize}
\end{lemma}

\begin{proof}
(a):
$$a\oplus a\diagup b\oplus b^{\sim}=1$$
$$a\diagup b\oplus b^{\sim}=a\diagup 1=a^{\sim}$$
$$a\diagup b=a^{\sim}\diagdown b^{\sim}$$
(b):
$$b^-\oplus b\diagdown a\oplus a=1$$
$$b^-\oplus b\diagdown a=1\diagdown a=a^-$$
$$b\diagdown a=b^-\diagup a^-$$
(c):
$$(a\oplus b)^-\oplus a\oplus b=1$$
$$(a\oplus b)^-\oplus a=1\diagdown b=b^-$$
$$(a\oplus b)^-=b^-\diagdown a$$
(d):
$$a\oplus b\oplus(a\oplus b)^{\sim}=1$$
$$b\oplus(a\oplus b)^{\sim}=a\diagup 1=a^{\sim}$$
$$(a\oplus b)^{\sim}=b\diagup a^{\sim}$$
\end{proof}

\begin{definition}\label{de:oplus} Let $(P;^-,^{\sim},\circ,\ast,0,1)$ be a pseudo Sasaki algebra.
(i) For $a\leq b^-$ (or equivalently $b\leq a^{\sim}$):\\
$a\oplus b:= (a^{\sim}\circ b^{\sim})^-=(b^-\ast a^-)^{\sim}$

(ii) For $a\leq b$:\\
$b\diagdown a:=a^-\circ b$, $a\diagup b:=a^{\sim}\ast b$
\end{definition}

\begin{lemma}\label{le:subtract} If $a\leq b$, then $b=(a^-\ast(a^-\circ b)^-)^{\sim}=b\diagdown a \oplus a$, and $b=(a^{\sim}\circ (a^{\sim}\ast b)^{\sim})^-=a\oplus (a\diagup b)$.
\end{lemma}
\begin{proof} If $a\leq b$, then by axiom (1), $b^-\leq a^-$, and $b^-\leq b^-$. By divisibility, $b^-=a^-\ast(a^-\circ b)^-$. Similarly, from $b^{\sim}\leq a^{\sim}$, $b^{\sim}\leq b^{\sim}$ we obtain $b^{\sim}=a^{\sim}\circ(a^{\sim}\ast b)^{\sim}$. The rest follows by (1) and Definition \ref{de:oplus} (ii) and (i), taking into account that $a^-\circ b\leq a^-, a^{\sim}\ast b\leq a^{\sim}$.
\end{proof}

\begin{theorem}\label{th:PSA-PEA}  With the operation $\oplus$ from Definition \ref{de:oplus},
a pseudo Sasaki algebra $(P;\oplus,0,1)$ is a lattice ordered PEA.
\end{theorem}

\begin{proof}
Let $a,b,c\in P$.\\
(PE1): by axiom (5) we have $(a\oplus b)\oplus c=((a\oplus b)^{\sim}\circ c^{\sim})^-=((a^{\sim}\circ b^{\sim})\circ c^{\sim})^-=(a^{\sim}\circ (b^{\sim}\circ c^{\sim}))^-=(a^{\sim}\circ (b\oplus c)^{\sim})^-=a\oplus (b\oplus c)$.\\
(PE2): At first we prove that $a\oplus a^{\sim}=a^-\oplus a=1$. Indeed, $a\oplus a^{\sim}=(a\ast a^-)^{\sim}$ and $(a\ast a^-)^{\sim}=1$ iff $a\ast a^-=0$ and this is ensured by Lemma \ref{le:1} (a). Similarly $a^-\oplus a=(a\circ a^{\sim})^-=1$. Now we show that if $a\oplus b=1$ then $b=a^{\sim}$ and if $b\oplus a=1$ then $b=a^-$. So let us have $a\oplus b=1, a\leq b^-$. Then $1=a\oplus b=(a^{\sim}\circ b^{\sim})^-$ so $a^{\sim}\circ b^{\sim}=0$ and by Lemma \ref{le:1} (c), $a^{\sim}\leq b$. Together with $a\leq b^-$ ($b\leq a^{\sim}$) we get $a^{\sim}=b$. If now $b\oplus a=1$, $b\leq a^-$, then $1=b\oplus a=(a^-\ast b^-)^{\sim}$, and so $a^-\ast b^-=0$. By Lemma \ref{le:1} (c), $a^-\leq b$, therefore $a^-=b$.\\
(PE 3):  We have $c:=a\oplus b=(a^{\sim}\circ b^{\sim})^-\geq a$ and $a\oplus b=(b^-\ast a^-)^{\sim}\geq b$ by Lemma \ref{le:1} (d). Applying Lemma \ref{le:subtract} yields $c=b\oplus b\diagup c$ and
$c=c\diagdown a\oplus a$.  Putting $d:=b\diagup c$ and $e:=c\diagdown a$ yields the existence of the elements $d$ and $e$. To show uniqueness, we consider $d,d_1$ such that $d\oplus a=d_1\oplus a$. Then $(d\oplus a)\oplus f=(d_1\oplus a)\oplus f=1$ for some $f\in P$ and by associativity $d\oplus (a\oplus f)=d_1\oplus (a\oplus f)=1$. By uniqueness of the left complement ((PE2)) we get $d=d_1$ and similar result we obtain for the right complement.\\
(PE4): If $a\oplus 1$ or $1\oplus a$ exist, then by the definition of $\oplus$ (Def. \ref{de:oplus}) $a\leq 1^-$ or $a\leq 1^{\sim}$.\\
Now we have a PEA and we moreover show that it is a lattice. Divisibility yields the existence of infima.
Using pseudoinvolution, we obtain $a\vee b=(a^-\wedge b^-)^{\sim}=(a^{\sim}\wedge b^{\sim})^-$.
\end{proof}

\section{From LPEA to pseudo Sasaki Algebra}

The following two operations on a lattice pseudo-effect algebra $(P;\oplus,0,1)$ were introduced in \cite{PuSas} as a generalization of the Sasaki product in LEAs:

\begin{definition}\label{de:circast}
$a\circ b:=a\wedge b^-\diagup a$; $\;\; a\ast b:=a\diagdown a\wedge b^{\sim}$
\end{definition}

\begin{theorem}\label{th:plus}
Under the previous definition we can write $a\oplus b$ as $(b^-\ast a^-)^{\sim}$ or $(a^{\sim}\circ b^{\sim})^-$.
\end{theorem}

\begin{proof}
If $a\oplus b$ exists, then $a\leq b^-$ and $b\leq a^{\sim}$. By Lemma \ref{le:minus},\\
$$(a\oplus b)^-=b^-\diagdown a=b^-\diagdown a\wedge b^-=b^-\ast a^-,$$
hence $a\oplus b=(b^-\ast a^-)^{\sim}$.

Similarly $(a\oplus b)^{\sim}=b\diagup a^{\sim}=b\wedge a^{\sim}\diagup a^{\sim}=a^{\sim}\circ b^{\sim}$, hence $a\oplus b =(a^{\sim}\circ b^{\sim})^-$.

\end{proof}

\begin{lemma}
$b\leq c$ implies $a\circ b\leq a\circ c$ and $a\ast b\leq a\ast c$.
\end{lemma}

\begin{proof}
$b\leq c$ implies $c^-\leq b^-$ and $c^{\sim}\leq b^{\sim}$. Then also $a\wedge c^-\leq a\wedge b^-$ and $a\wedge c^{\sim}\leq a\wedge b^{\sim}$. Therefore $a\circ b=a\wedge b^-\diagup a\leq a\wedge c^-\diagup a=a\circ c$ and $a\ast b=a\diagdown a\wedge b^{\sim}\leq a\diagdown a\wedge c^{\sim}=a\ast c$.
\end{proof}

\begin{theorem}\label{th:PEA-PSA} Let $(P;\oplus,0,1)$ be a lattice pseudo-effect algebra. With operations $\circ$ and $\ast$ defined as in Definition \ref{de:circast}, $(P;^-,^{\sim},\circ,\ast,0,1)$ becomes a pseudo Sasaki algebra.
\end{theorem}
\begin{proof} We have to prove axioms (1) -(6) of pseudo Sasaki algebra.
(1): This follows from the definitions of complements and partial order in PEA.\\
(2): $$a\circ 1=a\wedge 1^-\diagup a=a\wedge 0\diagup a=a$$
$$1\circ a=1\wedge a^-\diagup 1=a^-\diagup 1=a$$
$$a\ast 1=a\diagdown a\wedge 1^{\sim}=a\diagdown a\wedge 0=a$$
$$1\ast a=1\diagdown 1\wedge a^{\sim}=1\diagdown a^{\sim}=a$$
(3): Let us first have $a\circ b\leq c$. Then $c^-\leq (a\circ b)^-=(a\wedge b^-\diagup a)^-$. By previous lemma we also have $a\ast c^-\leq a\ast (a\wedge b^-\diagup a)^-=a\diagdown(a\wedge b^-\diagup a)=a\wedge b^-\leq b^-$. On the other hand if $a\ast c^-\leq b^-$ and therefore $b\leq (a\ast c^-)^{\sim}$, then $a\circ b\leq a\circ (a\ast c^-)^{\sim}=a\circ (a\diagdown a\wedge c)^{\sim}=(a\diagdown a\wedge c)\diagup a=a\wedge c\leq c$. The remaining equality can be proved in the same way.\\
(4): Let $c\leq a, c\leq b$. We show that $a\circ (a\ast b^-)^{\sim}=a\ast (a\circ b^{\sim})^-=a\wedge b$. Indeed, $a\circ (a\ast b^-)^{\sim}=a\circ (a\diagdown a\wedge b)^{\sim}=(a\diagdown a\wedge b)\diagup a=a\wedge b$ and $a\ast (a\circ b^{\sim})^-=a\ast (a\wedge b\diagup a)^-=a\diagdown (a\wedge b\diagup a)=a\wedge b$. Therefore $c\leq a\circ (a\ast b^-)^{\sim}=a\ast (a\circ b^{\sim})^-$.\\
(5): PEA is associative and so $a\oplus (b\oplus c)=(a\oplus b)\oplus c$. Then $a\oplus (b^{\sim}\circ c^{\sim})^-=(a^{\sim}\circ b^{\sim})^-\oplus c$ and $(a^{\sim}\circ (b^{\sim}\circ c^{\sim}))^-=((a^{\sim}\circ b^{\sim})\circ c^{\sim})^-$ so that finally $a^{\sim}\circ (b^{\sim}\circ c^{\sim})=(a^{\sim}\circ b^{\sim})\circ c^{\sim}$. Similarly we prove the second associativity condition.\\
(6) follows from Theorem \ref{th:plus}.
\end{proof}

\begin{remark}\label{re:pseudocom}

Two elements $a,b$ in a lattice pseudo-effect algebra are \emph{compatible} iff $(a\vee b)\diagdown a=b\diagdown (a\wedge b)$ and $(a\vee b)\diagdown b=a\diagdown(a\wedge b)$, or equivalently, iff $(a\wedge b)\diagup a=b\diagup(a\vee b)$ and $(a\wedge b)\diagup b=a\diagup(a\vee b)$ \cite[Proposition 3.6]{DvVe3}.

We say that elements $a$ and $b$ of a pseudo Sasaki algebra \emph{pseudocommute} ($a\leftrightarrow b$) if
$$b\ast a^-=a^-\circ b \mbox{ and } a\ast b^-=b^-\circ a$$
or equivalently
$$a\circ b^{\sim}=b^{\sim}\ast a \mbox{ and } b\circ a^{\sim}=a^{\sim}\ast b$$

\begin{theorem}
Elements $a,b$ in a lattice pseudo-effect algebra are compatible if and only if they pseudocummute in the corresponding pseudo Sasaki algebra.
\end{theorem}

\begin{proof} By Lemma \ref{le:minus}, $(a\vee b)\diagdown a=a^-\wedge b^-\diagup a^-=a^-\circ b$,
and $b\diagdown(a\wedge b)=b\ast a^-$. Similarly, $(a\wedge b)\diagup a=a\circ b^{\sim}$, while $b\diagup(a\vee b)=b^{\sim}\diagdown(a^{\sim}\wedge b^{\sim})=b^{\sim}\ast a$. From this we derive the desired equivalences.
\end{proof}

Non-commutative generalizations of MV-algebras were introduced in \cite{GeIo} as pseudo-MV-algebras, and in \cite{Rac} as generalized MV-algebras. These definitions are equivalent. In \cite[Theorem 8.7]{DvVe2}, it was shown that a pseudo MV-algebra is a lattice pseudo-effect algebra where all pairs of elements are  compatible.
This gives the following characterization of pseudo-MV-algebras.

\begin{theorem}
A Sasaki algebra (LPEA) $P$ is a pseudo MV-algebra, if $a\ast b=b\circ a\;\; \forall a,b\in P$.
\end{theorem}

\end{remark}

\begin{remark}\label{re:weaklycom} There exist non-commutative lattice pseudo-effect algebras such that
$a^{\sim}=a^-$ \cite{Rac}. Such algebras are sometimes connected with cyclically ordered unital groups in the sense of Rieger \cite{Ri, Fuc}.  This class of LPEAs can be characterized by $a\rightarrow 0=a\leadsto 0$.

Notice that a lattice effect algebras can be characterized by $a\rightarrow b=a\leadsto b$ whenever $b\leq a$. Indeed, then $(a\circ b^{\sim})^-=(a\ast b^-)^{\sim}$, whence $a^-\oplus b=(a^{- \sim}\circ b^{\sim})^-=(a^{\sim -}\ast b^-)^{\sim}=b\oplus a^{\sim}$, and $a^-=a\rightarrow 0=a\leadsto 0=a^{\sim}$.
\end{remark}

\section{Pseudo-effect algebras as conditional double CI-posets}

\begin{definition}\label{de:cdci} An algebraic system $(R;\circ, \ast,\rightarrow, \leadsto,0,1)$ is called a \emph{conditional double CI-poset} iff the following axioms are satisfied:
\begin{enumerate}
\item[(a)] $x\rightarrow y$ is defined iff $y\leq x$\\
$x\leadsto y$ is defined iff $y\leq x$\\
$x\circ y$ is defined iff $y\rightarrow 0\leq x$\\
$x\ast y$ is defined iff $y\leadsto 0\leq x$
\item[(b)] $x\circ 1=1\circ x=x\ast 1=1\ast x=x$
\item[(c)] if $x\circ y$ and $x\leadsto z$ is defined, then $x\circ y\leq z$ iff $y\leq x\leadsto z$\\
if $x\ast y$ and  $x\rightarrow z$ is defined, then $x\ast y\leq z$ iff $y\leq x\rightarrow z$
\end{enumerate}

We say that a conditional double CI-poset satisfies
\begin{enumerate}
\item[(d)] \emph{pseudoinvolution} iff $x\leq y$ implies  $y^-\leq x^-$ and  $y^{\sim}\leq x^{\sim}$  and $x^{- \sim}=x=x^{\sim -}$,
where $x^-=x\rightarrow 0$, $x^{\sim}:=x\leadsto 0$
\item[(e)] \emph{divisibility} iff $x\leq y$ $\Leftrightarrow$  $x=y\circ (y\leadsto x)=y\ast(y\rightarrow x)$
\item[(f)] \emph{associativity} iff $(x\circ y)\circ z=x\circ (y\circ z)$, $(x\ast y)\ast z=x\ast (y\ast z)$  in the sense that if one side is defined so is the other and equality holds
\item[(g)] \emph{pseudo-effect algebra condition} iff $(y^-\ast x^-)^{\sim}=(x^{\sim}\circ y^{\sim})^-$
\end{enumerate}
\end{definition}

\begin{theorem}\label{th:CI-pea} Let $(R;\circ, \ast,\rightarrow, \leadsto,0,1)$ be a conditional CI-poset satisfying additional conditions (d)-(g). Define $x\oplus y=(y^-\ast x^-)^{\sim}=(x^{\sim}\circ y^{\sim})^-$, which is defined iff $x\leq y^-$. Then $P(R)=(R;\oplus, 0,1)$ is a pseudo-effect algebra. Moreover, the partial order induced by $\oplus$ coincides with the initial order $\leq$.
\end{theorem}

\begin{proof} To prove (PE1), assume that $x\oplus y$ and $(x\oplus y)\oplus z$ exist. Then $x\oplus y=(x^{\sim}\circ y^{\sim})^-$, and $(x\oplus y)\oplus z=((x^{\sim}\circ y^{\sim})\circ z^{\sim})^-=(x^{\sim}\circ(y^{\sim}\circ z^{\sim}))^-=x\oplus(y\oplus z)$ by (f).

(PE2): From (d) we easily obtain $1^{\sim}=1^-=0$ and $0^{\sim}=0^-=1$. So $1=x\oplus y=(x^{\sim}\circ y^{\sim})^-$ iff $x^{\sim}\circ y^{\sim}=0$. $x\oplus a$ is defined iff $a\leq x^{\sim}$, and $x\oplus a=1$ iff $x^{\sim}\circ a^{\sim}\leq 0$, which by residuation holds iff $a^{\sim}\leq x^{\sim}\leadsto 0=x^{\sim \sim}$, whence $x^{\sim}\leq a$. It follows that the unique element $a$ such that $x\oplus a=1$ is $a=x^{\sim}$.

Similarly, $b\oplus x$ is defined iff $b\leq x^-$, and $b\oplus x=1$ means that $x^-\ast b^-\leq 0$, which by residuation holds iff $b^-\leq x^-\rightarrow 0=x^{- -}$, whence $x^-\leq b$. It follows that the unique element $b$ such that $b\oplus x=1$ is $b=x^-$.

(PE3): Let $a\oplus b=c$, then $1=c\oplus c^{\sim}=(a\oplus b)\oplus c^{\sim}=a\oplus(b\oplus c^{\sim})$, and $1=c^-\oplus c=c^-\oplus(a\oplus b)=(c^-\oplus a)\oplus b$ implies by (PE2) that $a\oplus b=c$ iff $a=(b\oplus c^{\sim})^-$, $b=(c^-\oplus a)^{\sim}$. So if $d\oplus a=c$, then $d=(a\oplus c^{\sim})^-$, and if $b\oplus e=c$, then $e=(c^-\oplus b)^{\sim}$. It remains to prove that such elements $d$ and $e$ are defined. In the same way as in Lemma \ref{le:2.3} (vii), we show that
$a\circ b\leq a$ and $a\ast b\leq a$. From $a\oplus b=(a^{\sim}\circ b^{\sim})^-$ we obtain that $a^{\sim}\circ b^{\sim}\leq a^{\sim}$, hence
$a\leq (a^{\sim}\circ b^{\sim})^-=a\oplus b=c$, and from $a\oplus b=(b^-\ast a^-)^{\sim}$ we get $b\leq a\oplus b=c$. It follows that the elements $d$ and $e$ are defined.

(PE4): $a\oplus 1$ is defined iff $a\leq 1^-=0$, hence $a=0$. Similarly, $1\oplus a$ is defined iff $1\leq a^-$, hence $a\leq 1^{\sim}=0$.

It remains to prove that $a\leq b$ in  $R$ iff $b=a\oplus c$, or $b=c\oplus a$ for some $c\in R$. We have already proved that $a\leq a\oplus c$ as well as $a\leq c\oplus a$. So assume that $a\leq b$. Then $b^-\leq a^-$, and divisibility implies $b^-=a^-\ast(a^-\rightarrow b^-)$, so that $b=(a^-\ast (a^-\rightarrow b^-))^{\sim}=(a^-\rightarrow b^-)^{\sim}\oplus a$. On the other hand, $a\leq b$ entails $b^{\sim}\leq a^{\sim}$, and by divisibility, $b^{\sim}=a^{\sim}\circ (a^{\sim}\leadsto b^{\sim})$, so that $b=(a^{\sim}\circ(a^{\sim}\leadsto b^{\sim}))^-=a\oplus(a^{\sim}\leadsto b^{\sim})^-$.
\end{proof}

\begin{theorem}
Let $(P;\oplus,0,1)$ be a pseudo-effect algebra. Define $x\circ y=y^-\diagup x$ for $y^-\leq x$ and $x\ast y=x\diagdown y^{\sim}$ for $y^{\sim}\leq x$. Furthermore define $x\rightarrow y=(x\circ y^{\sim})^-$ and $x\leadsto y=(x\ast y^-)^{\sim}$, both for $y\leq x$. Then $(P;\circ,\ast,\rightarrow,\leadsto,0,1)$ is a conditional double CI-poset satisfying (d)-(g).
\end{theorem}

\begin{proof}
First we express binary relations $\circ$ and $\ast$ through $\oplus$: by Lemma \ref{le:minus} $\, x\circ y=y^-\diagup x=(x^-\oplus y^-)^{\sim}$ and $x\ast y=x\diagdown y^{\sim}=(y^{\sim}\oplus x^{\sim})^-$. Now we may also write $x\rightarrow y=x^-\oplus y$ and $x\leadsto y=y\oplus x^{\sim}$.
Next we show that if $b\leq c$ and $a\circ b$, $a\circ c$ exist, then $a\circ b\leq a\circ c$. Indeed, $b\leq c\,\Rightarrow\, c^-\leq b^-\,\Rightarrow\, a^-\oplus c^-\leq a^-\oplus b^-\,\Rightarrow\, (a^-\oplus b^-)^{\sim}\leq (a^-\oplus c^-)^{\sim}$. Similarly this monotone property holds for $\ast$. \\
(b): $x\circ 1=(x^-\oplus 1^-)^{\sim}=x^{-\sim}=x$; $1\circ x=(1^-\oplus x^-)^{\sim}=x$; $x\ast 1=(1^{\sim}\oplus x^{\sim})^-=x$ and $1\ast x=(x^{\sim}\oplus 1^{\sim})^-=x$.\\
(c): Let $x\circ y\leq z$. Then $(x^-\oplus y^-)^{\sim}\leq z\,\Rightarrow\, z^-\leq x^-\oplus y^-$ and $x\ast z^-\leq x\ast (x^-\oplus y^-)=((x^-\oplus y^-)^{\sim}\oplus x^{\sim})^-=(y^-\diagup x\oplus x^{\sim})^-$. But we know that $y^-\oplus y^-\diagup x\oplus x^{\sim}=1$, thus $y^-\diagup x\oplus x^{\sim}=y$ and $(y^-\diagup x\oplus x^{\sim})^-=y^-$. Thus we have $(z\oplus x^{\sim})^-=x\ast z^-\leq y^-$ and therefore $y\leq z\oplus x^{\sim}=x\leadsto z$.\\
Conversely, let $y\leq x\leadsto z=(x\ast z^-)^{\sim}=(x\diagdown z)^{\sim}$. Then $x\circ y\leq x\circ (x\diagdown z)^{\sim}=(x\diagdown z)\diagup x=z$. The second statement can be proved similarly.\\
(d) Here the only thing to prove is that $x\rightarrow 0=x^-$ and $x\leadsto 0=x^{\sim}$. But $x\rightarrow 0=x^-\oplus 0=x^-$ and $x\leadsto 0=0\oplus x^{\sim}=x^{\sim}$.\\
(f) Associativity of $\circ$ and $\ast$ is implied by associativity of $\oplus$.\\
To prove (g) it is enough to express $x\oplus y$ in terms of conditional double CI-poset operations and we obtain $(x^{\sim}\circ y^{\sim})^-=x\oplus y=(y^-\ast x^-)^{\sim}$.\\
(e): Let $x\leq y$. Compute: $y\circ (y\leadsto x)=y\circ (x\oplus y^{\sim})=(y^-\oplus (x\oplus y^{\sim})^-)^{\sim}$. Now $(x\oplus y^{\sim})^-\oplus x\oplus y^{\sim}=1$, thus $(x\oplus y^{\sim})^-\oplus x=y$ and so finally by Lemma \ref{le:minus} $ x=(y^-\oplus (x\oplus y^{\sim})^-)^{\sim}$. Similarly $y\ast (y\rightarrow x)=y\ast (y^-\oplus x)=((y^-\oplus x)^{\sim}\oplus y^{\sim})^-$ and while $y^-\oplus x\oplus (y^-\oplus x)^{\sim}=1$, we get $y=x\oplus (y^-\oplus x)^{\sim}$, thus $x=(y^-\oplus x)^{\sim}\oplus y^{\sim})^-$.
\end{proof}

\begin{remark} In \cite{ChaHa},  pseudo-effect algebras  satisfying the additional identity
$(x^-\oplus y^-)^{\sim}=(x^{\sim}\oplus y^{\sim})^-$   are called \emph{good pseudo-effect algebras}, and it was shown that they can be characterized by means of so-called  conditionally residuated structure $({\mathcal R}=(R;\leq, .,\rightarrow,\leadsto,0,1)$, which is a bounded poset ($0$ is the least and $1$ is the greatest element) with three binary operations $., \rightarrow,\leadsto$, where $.$ and $\rightarrow$ and $.$ and $\leadsto$ are related by residuation.

In terms of conditional double CI-posets, good pseudo-effect algebras can be characterized by the additional identity
$x\circ y=y\ast x$, whenever $y^-\leq x$.

\end{remark}


\begin{thebibliography}{XX}
\bibitem{BJ} T.S. Blyth, M.F. Janowitz: {Residuation Theory}, International Series of Mpnographs in Pure and Applied Mathemetics, Vol. 102, Pergamon Press, Oxford-New York-Toronto, 1972.
\bibitem{ChaHa} I. Chajda, R. Hala\v s: {\it Effect algebras are conditionally residuated structures}, Preprint.
\bibitem{ChaKu} I. Chajda, J. K\"uhr: {\it Pseudo-effect algebras as total algebras}, Int. J. Theor Phys., to appear.
\bibitem{DvPu} A. Dvre\v censkij, S. Pulmannov\'a: {\it New Trends in Quantum Structures}, Kluwer Acad. Publ., Dordrecht/Boston/London, Inter. Sci., Bratislava, 2000.
\bibitem{DvVe1} A. Dvure\v censkij, T. Vetterlein: {\it pseudo-effect algebras. I. Basic properties} Inter. J. Theor. Phys. {\bf 40} (2001), 685 -- 701.
\bibitem{DvVe2} A. Dvure\v censkij, T. Vetterlein: {\it pseudo-effect algebras. II. Group representation.} Inter. J. Theor. Phys. {\bf 40} (2001), 703 -- 726.
\bibitem{DvVe3} A. Dvure\v censkij, T. Vetterlein: {\it On pseudo-effect algebras which can be covered by pseudo MV-algebras.} Demonstratio Mathematica {\bf 36} (2003), 261 -- 282.
\bibitem{FoBe} D.J. Foulis, M.K. Bennett: {\it Effect algebras and unsharp quantum logics}, Found. Phys. {\bf 24} (1994), 1325--1346.
\bibitem{FoPu} D. J. Foulis, S. Pulmannov\'a: {\it Logical connectives on lattice effect algebras.}, preprint.
\bibitem{Fuc} L. Fuchs: {\it Partially Ordered Algebraic Systems}, Pergamon Press, Oxford, London, NY, Paris, 1963.
\bibitem{GeIo} G. Georgescu, A. Iorgulescu: {\it Pseudo-MV algebras}, Multi Val. Logic {\bf 6} (2001), 95--135.
\bibitem{PuSas} S. Pulmannov\'a: {Generalized Sasaki projections and Riesz ideals in pseudo-effect algebras}, Int. J. Theor. Phys. {\bf 42} (2003), 1413--1423.
\bibitem{Rac} J. Rach$\dot{u}$nek: {A non-commutative generalizatin of MV-algebras}, Czechoslovak Math. J. {\bf 52} (2002), 255--273.
\bibitem{Ri} L. Rieger:{ On the ordered and cyclically ordered groups I,II,III}, V\v est. Kr\'al. \v Cesk\'e Spol. Nauk (1946, 1947, 1948) (in Czech).

\end{thebibliography}
\end{document}